\documentclass[11pt]{amsart}
\usepackage{amsmath}
\usepackage{amssymb}
\usepackage{amsfonts}
\usepackage{amsthm}
\usepackage{verbatim}
\usepackage{amscd}
\usepackage{cite}
\usepackage{leftidx}
\usepackage{enumerate}
\usepackage{txfonts}
\usepackage{manfnt}
\usepackage{amscd}
\usepackage[mathscr]{eucal}
\usepackage{hyperref}
\usepackage{datetime2}
\usepackage{graphics}
\usepackage{mathrsfs}  
\usepackage{tikz-cd}
\usepackage{enumitem}
\setlist{nolistsep}
\usepackage{mathpazo}
\def\bP{\mathbb P}
\def\ar[r]{\to}

\def\cA{\mathcal A}
\def\cO{\mathcal O}

\newtheorem{thm}{Theorem} 

\newtheorem*{thm*}{Theorem}
\newtheorem*{prop*}{Proposition}
\newtheorem{cor}[thm]{Corollary}
\newtheorem*{cor*}{Corollary}

\newtheorem{lem}[thm]{Lemma}
\newtheorem*{lem*}{Lemma}

\newtheorem*{claim*}{Claim}
\newtheorem{prop}[thm]{Proposition}

\theoremstyle{remark}

\newtheorem{rem}[thm]{Remark}
\newtheorem*{rem*}{Remark}
\newtheorem{crit-rem}[thm]{Critical remark}
\newtheorem{remarks}[thm]{Remarks}

\newtheorem{example}[thm]{Example}
\newtheorem*{example*}{Example}
  
\newtheorem*{defn*}{Definition}
\newtheorem*{con*}{Conjecture}

 \DeclareMathOperator{\gr}{Gr}

 \DeclareMathOperator{\Hom}{Hom}
\DeclareMathOperator{\coker}{coker}

\def\cA{\mathcal A}
\def\cO{\mathcal O}
\def\bP{\mathbb P}

\def\cF{\mathcal{F}}

\def\refp #1.{(\ref{#1})}
\newcommand\carets [1]{\langle #1 \rangle}

\newcommand{\A}{\mathcal{A}}

\newcommand{\M}{\mathcal{M}}

\newcommand{\Cal}[1]{\mathcal #1}

\def\sbr #1.{^{[#1]}}
\def\sfl #1.{^{\lfloor #1\rfloor}}

\def\?{{\bf{??}}}

\def\Proj{\textrm{Proj}}

\def\M{\Cal M}
\def\A{\Bbb A}

\def\P{\mathbb P}

\def\R{\mathbb R}
\def\Z{\mathbb Z}

\def\sym{\text{\rm Sym} }

\def\O{\mathcal O}

\def\Sym{\textrm{Sym}}

\def\g{\mathfrak g}

\def\k{\mathfrak k}

\def\hom{\mathfrak {hom}}

\def\1/2{\frac{1}{2}}

\def\I{\mathcal{ I}}

\def\im{\text{im}}

\def\2{{[2]}}
\def\l{\ell}
\def\nl{\newline}

\def\hom{\mathcal{H}\mathit{om}}

\def\<{\langle}
\def\>{\rangle}

\def\im{\text{im}}
\def\2{{[2]}}
\def\l{\ell}

\def\Proj{\text{Proj}}
\def\scl #1.{^{\lceil#1\rceil}}
\def\spr #1.{^{(#1)}}
\def\sbc #1.{^{\{#1\}}}

\def\subpr#1.{_{(#1)}}

\def\beq{\begin{equation*}}
\def\eeq{\end{equation*}}

\newcommand{\mlog}[1]{\langle-\log {#1} \rangle}
\def\g3{{\Gamma\spr 3.}}

\newcommand{\eqspl}[2]{
\begin{equation}\label{#1}
\begin{split}
#2\end{split}\end{equation}}

\newcommand{\exseq}[3]{
0\to #1\to #2\to #3\to 0
}

\newcommand{\beginalphaenum}{
\begin{enumerate}\renewcommand{\labelenumi}{ }
\item \begin{enumerate}
}

\def\eex{\end{rm}\end{example}}

\begin{document} 
	\title{A family of  indecomposable rank-$n$ vector bundles on $\P^n\times\P^n$
	in positive characteristics }
	\author 
	{Ziv Ran and J\"urgen Rathmann}
%
%
	\thanks{arxiv.org 2503.24339  }
	\date {\DTMnow}
%
%
	\address {\nl ZR: UC Math Dept. \nl
	Skye Surge Facility, Aberdeen-Inverness Road
	\nl
	Riverside CA 92521 US\nl 
	ziv.ran @  ucr.edu\nl
	\url{https://profiles.ucr.edu/app/home/profile/zivran}\newline
	JR: jk.rathmann@gmail.com
	}
	
	 \subjclass[2010]{14n05, 14j60}
	\keywords{vector bundle, projective space, monad}
	\begin{abstract}
We construct a sequence of rank-$n$ indecomposable vector bundles on $\P^n\times\P^n$ 
for every $n\geq 2$ and in every positive characteristic $p$ that
are not pullbacks via any map  $\P^n\times\P^n\to \P^{m} $. We study in detail the case
$p=n=2$ where the bundles in question can be described by a monad.


		\end{abstract}
\maketitle	
\section*{Data availability statement}
The manuscript has no associated data.
\section*{Conflict of interest statement}
On behalf of all authors, the corresponding author states that there is no conflict of interest.
\section*{Introduction}
In 1973 Horrocks and Mumford \cite{horrocks-mumford} constructed an indecomposable rank  $r=2$ 
vector bundle on $\bP^4$, and up to standard modifications (twists and pullbacks 
by finite self-maps of projective space) this bundle remains the only known indecomposable
rank 2 bundle on $\bP^n,n\ge 4$ in characteristic 0.
Since then, there have been just a few other constructions of indecomposable bundles of
rank 2, including the Horrocks bundle on $\P^4$ in characteristic 2 \cite{horrocks-bundles-on-pn},
the Tango bundle on $\bP^5$ in characteristic 2  \cite{tango-morphism}, and several families on 
$\bP^4$ in positive characteristic and on $\bP^5$ in characteristic 2 constructed by 
Kumar, Peterson and Rao \cite{kumar-p4, kumar-peterson-rao}. In fact a conjecture of Hartshorne states that
no indecomposable rank 2 vector bundle exists on $\bP^n$, $n\ge 6$.

Even in the higher rank in the range  $3\le r\le n-2$, there are only few 
non-trivial (i.e., up to twists, finite pull-backs and extensions of other known bundles) 
constructions of indecomposable vector bundles of rank $r$ on any projective space of 
dimension $n\ge 5$: 
Horrocks (cf. \cite{horrocks}, \cite{decker-manolache-schreyer}, \cite{ancona-ottaviani}) constructed an indecomposable
rank-3 bundle on $\P^5$
(in any characteristic) whose reduction 
in characteristic 2 has Tango's rank 2 bundle as a subbundle, and Kumar and his coauthors
exhibited several rank-3 examples on $\bP^5$ in characteristic 2 \cite{kumar-peterson-rao}.

The picture improves somewhat for $r=n-1$, where we know a family of bundles 
constructed by Tango \cite{tango-morphism}, the null-correlation bundles
(for odd $n$), the Sasakura bundle on $\bP^4$ \cite{a-d-sasakura},
and several families of rank 3 bundles on $\bP^4$ (any characteristic) in \cite{kumar-peterson-rao}.

In addition, one can construct many rank $n-1$ bundles on $\bP^{n}$ for odd $n$ using
$L$-symplectic structures for some line bundle $L$, as follows:
\begin{enumerate}
\item[--]
A globally generated $L$-symplectic bundle $E$ of
rank $r$ has a trivial line subbundle $S$ \cite[4.3.1]{okonek}, 
and the symplectic structure provides by duality a 
line bundle quotient of $E/S$, 
thus exhibiting an $L$-symplectic bundle of rank $r-2$ as a subquotient.
\item[--]
Given an arbitrary $L$-symplectic bundle $E$, the twist $E\otimes L(m)$ will be globally generated 
and $L(2m)$-symplectic for any large $m$. 
\item[--]
By repeatedly applying the first two steps (possibly after first adding copies of the trivial $L$-symplectic
bundle $\O\oplus L$ if the rank of $E$ is less than $n$), one obtains an $L'$-symplectic subquotient of 
rank $k\leq n$ for some line bundle $L'$.
\item[--]
Since the rank of any $L$-symplectic bundle is even and the reduction step reduces the rank by $2$, 
the final subquotient must have rank $k\leq n-1$, if $n$ is odd.
\end{enumerate}
\medskip

Most of these constructions date from the 1970s and early 1980s, and the status at that time
is summarized in the book by Okonek, Schneider and Spindler \cite{okonek}.
Later years have seen only marginal progress, with the notable exception of the work of Kumar, Peterson and Rao.\par
Next to projective spaces one might naturally inquire about products of such.
In that vein, the only work we are aware of is that of Solis \cite{solis-p1p1}
on rank-2 vector bundles on $\P^1\times\P^1$.\par

\par
The purpose of this paper is to construct some  indecomposable rank-$n$ vector bundles
$E$ on  $\P^n\times\P^n$ in every positive characteristic  $p$, with special attention
to the case $p=n=2$.
In a separate paper
\cite{p2p2-II} we study zero sets of the latter bundles in some detail and
relate them to  the geometry of nonclassical Enriques surfaces in $\P^2\times\P^2$.
\par 
To state our  general results we use the following notation. $\P^n_L, \P^n_h$
are copies of $\P^n$ with hyperplane classes $L$ resp. $h$, 
$p_L:\P^n_L\times\P^n_h\to\P^n_L$ is the projection and ditto $p_h$,  
$Q_L, Q_h$ are the respective
 tautological quotient bundles and $F$ is the Frobenius endomorphism.
 We also recall that a \emph{monad} is a complex of locally free sheaves of the form
 \[0\to A\to B\to C\to 0\]
 which is exact except possibly at the $B$ term, and such that the map $A\to B$
 embeds $A$ as a subbundle.  
Now our main  result is the following.. 
\begin{thm*}\label{main-thm}[Main Theorem] (i)
For every $n\geq 2$, $p>0$ prime, $q=p^a, 1\leq k\leq q$,  and
every smooth divisor $\cA\in |L+h|$ on $\P^n_L\times\P^n_h$, 
the kernel  $E$ of the canonical map
\[p_h^*F^{a*}Q_h\to\O_{k\cA}(qL)\]
is a
 rank-$n$ vector bundle on $\P^n_L\times\P^n_h$
that is indecomposable and  not a twist of a pullback from any $\P^m$.\par
(ii) Moreover in case $p=n=2$, a suitable twist of $E$ is the cohomology of a monad of the form
\[\O\to Q_L\otimes Q_h\to \O(L+h).\]
Conversely the cohomology bundle of such a monad is one of the bundles $E$ above.
\end{thm*}
Some closely related constructions were implicitly contained in a paper by
Lauritzen and Rao \cite{lauritsen-rao}.\par
As this result suggests in case $k=1$, there is in fact a bijection between bundles $E$ and divisors
$\cA$ which extends to deformations. Hence the moduli space of $E$ may
be identified with the open set of smooth divisors in the linear system $|L+h|=\P^{n^2+2n}$. \par
The case $n=q=2, k=1$ will be studied further in our paper
\cite{p2p2-II}.  There we obtain more detailed results about the bundle $E$ 
and its zero-sets, as follows:
\begin{itemize}
\item  The zero set $Y$ of a general section of a suitable  twist of $E$ is a 
smooth Enriques surface
in characteristic 2 with trivial canonical bundle and irregularity 1 (such surfaces
are usually called nonclassical). The general smooth zero set is 
ordinary, i.e. singular in the sense of Bombieri-Mumford 
(Frobenius acts isomorphically on $H^1(\O_Y)$) and there is a 
codimension-1 subfamily of smooth supersingular zero sets (Frobenius is zero on $H^1(\O_Y)$).
\item  The same twist has special sections with zero-set of the form $Y_1\cup_WY_2$ where $Y_1\simeq \P^2,$
$Y_2$ is an elliptic ruled surface and $W$ is a smooth   elliptic curve
\end{itemize}
\bigskip

The paper is organized in 2 parts. First in Part 1, \S 1 establishes some preliminaries and notation. 
The main construction is presented in \S 2, which also studies the restriction of the bundle $E$
on the divisor $\cA$ and on lines contained or not contained in $\cA$ and proves 
an indecomposability and strong nondegeneracy property  for $E$. The cohomology of twists of
$E$ is studied in \S 3. In \S 5 we establish a symmetry property of the divisor $\cA$
which shows that even though the construction of $E$ is asymmetric with respect to
interchanging the factors, the divisor $\cA$ in fact depends symmetrically on $E$.
Part 2, about the case $p=n=2$, 
establishes the equivalence of the construction via elementary modification along $\cA$
and another construction via a monad. These results play an important role in \cite{p2p2-II}. 
Going from modification to monad, the argument is based on a relative Beilinson monad.
In the opposite direction the argument is based on studying jumping lines in fibres.

\subsection*{Acknowledgment} We are  grateful to Gwoho Liu for his assistance
with numerical experiments that ultimately led to (the Chern classes of) $E$.
\subsection*{Conventions} We work over a fixed algebraically closed field $\k$, 
usually of characteristic $p>0$.
The dual of a vector bundle $E$ is denoted by $E^*$ or $\check E$

\section{Preliminaries: symmetric, exterior, divided powers}\label{divided-sec}%
Let $\k$ be an algebraically closed field of characteristic $p>0$,
and let $X$ be a $\k$-scheme.

The (absolute) Frobenius is a morphism $F\colon X\to X$ of schemes (but not of $k$-schemes)
which is the identity on the topological space $X$
and where $F^\#\colon \O_X\to \O_X$ is the $p$th power map.
$F$ is a finite, flat, purely inseparable morphism.

The pullback of sheaves under Frobenius is right-exact (as is any pull-back functor), 
and we would like to note the following:
\begin{enumerate}
\item
If $\cF$ is locally free on $X$ and defined by transition matrices $A_{ji}$ with respect to a
covering $(U_i)$ of $X$ (where $F$ is free on $U_i$), then $F^*\cF$ is defined by
transition matrices $A^{(p)}_{ji})$ where each entry from $A_{ji}$ has been raised to the 
$p$th power.
\item
If $L$ is a line bundle on $X$, then as sheaves, $F^*L\cong L^{\otimes p}$.
\end{enumerate}

See, e.g. \cite{eisenbud-view}, Appendix 2, for details.\par
The (graded) symmetric algebra $\sym^\bullet M$ on a module $M$ is obtained from the 
graded tensor algebra 
by factoring out the homogeneous ideal generated by
$\{x\otimes y - y\otimes x , x, y\in M\}$, whereas the exterior algebra $\wedge^\bullet M$ is obtained by factoring out
the ideal generated by $\{x\otimes x, x\in M\}$.

In characteristic $2$, we have 
\[(x+y)\otimes (x+y)-x\otimes x-y\otimes y
=x\otimes y-y\otimes x,\] 
hence the ideal defining $\wedge^\bullet M$ contains the ideal defining $\sym^\bullet M$,
so the exterior algebra is a quotient of the symmetric algebra.\par
For a locally free module $M$, we denote by $D^\bullet M$ its graded 
divided power algebra (see \cite{eisenbud-view} or \cite[3.9]{berthelot-ogus}). It is 
generated by sections of the form $a^i/i!, a\in M$.

Key property of the divided power algebra is that
\begin{equation*}
D^k(M)=\big(\Sym^k(M^\vee)\big)^\vee.
\end{equation*}

\begin{lem}
Let $M$ be a locally free sheaf on a scheme over a field $\k$ of characteristic $p\geq 0$.
\begin{enumerate}
\item[\textup{1.}]
If $p>0$, there is an inclusion $F^*M\to\sym^pM$ whose image consists of the $p$th powers.
\item[\textup{2.}]
There is an exact sequence
\begin{equation*}
0\to \wedge^2M \to M\otimes M \to \sym^2M \to 0
\end{equation*}
which splits if $p\neq 2$.
\item[\textup{3.}]
If $p=2$, then we have exact sequences
\begin{equation*}
0 \to F^*M \to \sym^2M \to \wedge^2 M \to 0
\end{equation*}
and
\begin{equation*}
0\to \wedge^2M \to D^2M \to F^*M \to 0.
\end{equation*}
\end{enumerate}
\end{lem}
All these statements are presumably well known, but there does not seem
to be an easily accessible reference for the third statement.

\begin{proof}
Statement 1 is obvious.\par
Regarding the first exact sequence from statement 3:
Assume first that $M$ is free over $U$ with basis $e_1,\dotsc,e_n$. Then
$\sym^2M$ (as quotient of $M\otimes M$) has the images of $(e_i\otimes e_j)_{i\le j}$ as a basis, and
the basis for $\wedge^2M$ is given by the images of $(e_i\otimes e_j)_{i< j}$.
Therefore the kernel of the map $\sym^2M\to \wedge^2M$ is also free and has the 
images of $(e_i\otimes e_i)_{1\le i\le n}$ as basis.

Now consider two open sets $U$, $V$ so that $M$ is free over each of them. The
transition matrix for the kernel bundle arises from the transition matrix for $M$
by squaring all entries.

Globally, this means that the kernel bundle is isomorphic to the Frobenius pullback of $M$.

The second exact sequence follows from the first by dualizing, replacing the dual of $M$
by $M$, and noting that exterior powers and Frobenius pullback commute with taking
duals.
\end{proof}
\begin{example}\label{divided-example}
Taking $M=Q$, the tautological quotient bundle on $\P^2$, we get exact
\[ \exseq{\O(1)}{ D^2Q}{F^*Q}\]
Now $F^*Q$ fits in an exact sequence (in fact, a minimal graded free resolution)
\[\exseq{\O(-2)}{3\O}{F^*Q}\]
hence \[h^0(F^*Q)=3,  h^1(F^*Q)=h^0(F^*Q(-1))=0, h^1(F^*Q(-1))=1.\]
This implies
\[h^0(D^2Q)=6, h^0(D^2Q(-1))=1.\]It follows that in char. 2, there is a unique map up to scalars
\[\sym^2Q\to\O(1)\]
and its kernel is $F^*Q$ which is thus uniquely embedded in $\sym^2Q$ as the set of squares.\par
Next, since $H^1\big(\O(\l)\otimes\O(1)\big)=0$ for any $\l \geq 0$, we have a short exact sequence
of graded modules
\begin{equation*}
0\to \bigoplus\limits_{\l\geq -1} H^0 \O(1+\l) \to \bigoplus\limits_{\l\geq -1} H^0 (D^2Q)(\l) \to 
\bigoplus\limits_{l\geq -1} H^0 (F^*Q)(\l) \to 0
\end{equation*}
Then the long exact sequence in Koszul cohomology \cite[1.d.4]{green-koszul} can be used to
obtain  the following minimal graded free  resolution of $D^2Q$:
\begin{equation*}
0\to \O(-2) \to \O(1)\oplus 3\O \to D^2Q \to 0
\end{equation*}
\end{example}
\section{Modifications}\label{modifications}
The following result is standard folklore (see e.g. \cite{huybrechts-lehn}[Prop. 5.2.2]):
\begin{prop}
Let $X$ be a scheme and $D\subset X$ a (possibly non-reduced) Cartier divisor.
Let $E$ be a vector bundle on $X$ and $F$ a locally $\O_D$-free quotient
of $E\otimes\O_D$. Then the kernel $E'$ of the natural surjection $E\to F$ is
$\O_X$-locally free. 
\end{prop}
\begin{proof}
The assertion is evidently local, so we can assume that $E$ and $F$ are free. 
Working locally, let $e_1,...,e_k$ be sections of $E$
mapping to a basis of $F$, let $y$ be an equation for $D$ and let $e_1, ...,e_k, ..., e_n$
be a basis of $E$ extending $e_1,...,e_k$. Then $ye_1,...,ye_k, e_{k+1}, ...,e_n$
is a basis for the locally free subsheaf $E'$.
\end{proof}

 \part{General case}
 \section{Construction}
 In this Part we work in characteristic $p>0$ and construct some rank-$n$ bundles on 
 $\P^n\times\P^n$ as described in the introduction, thus proving part (i) of the Main Theorem.
 The construction is carried out in 
 \S\S \ref{construction-base}, \ref{construction-q}, \ref{construction-qk}.
 In \S \ref{nondegeneracy} we show in some cases that our bundle is not a pullback
 from lower dimension, and in \S\ref{indecomposability} we show it is indecomposable.
 The remaining sections, which are needed in the further study of the bundle,
  study restriction on the divisor $\cA$ and jumping lines.
 \subsection{Construction: base case}\label{construction-base}
 Assume the ground field $\k$ has characteristic $p>0$ and let $F$ be the Frobenius endomorpism
 on $\P^n_h$. By \S\ref{divided-sec}, we have an  inclusion of sheaves
 \[F^*Q_h\to\sym^p Q_h\] 
 and an inclusion
 \[\P(Q_h)\to\P^n_L\times\P^n_h.\]
 Choose an embedding $\P(Q_h)\to\P^n_L\times \P^n_h$,
 and denote its image  $\cA$, which is a divisor of type $L+h$. Then we get a map, clearly surjective
 \[p_2^*F^*Q_h\to\O_{\cA}(pL),\]
 whose image on $H^0$  is the set of all $p$-th powers.
 The kernel is a rank-$n$ vector bundle on $\P^n_L\times \P^n_h$ which we denote by $E_0(-L)[n]$
 or, when $n$ is understood, by $E_0(-L)$. Thus we have exact
 \eqspl{e0-def-modif}{
 \exseq{E_0(-L)}{h^*F^*Q_h}{\O_\cA(pL)}.
 }
 \par
 From \eqref{e0-def-modif} we can compute the  total Chern class
 \[c(E_0(-L))=\frac{1+(p-1)L-h}{(1-ph)(1+pL)}.\]
 In particular for $n=2$ we have
 \[c(E_0(-L))=1+(p-1)h-L+p(p-1)h^2+pL^2,
 \]
 hence
 \eqspl{chern-n=2}{c(E_0)=1+(p-1)h+L+p(p-1)h^2+(p-1)hL+pL^2.}
 In particular, for $p=n=2$ this Chern class is symmetric.
\subsection{Construction: $q$ case}\label{construction-q} 
More generally, for any $a\geq 1, q=p^a$ we get a bundle denoted $E_0[n,q](-L)$ as the kernel of the surjection
 \eqspl{q-case}{
 p_h^*(F^a)^*(Q_h)\to\O_\cA(qL).
 }
 Thus we have exact
 \eqspl{modif-q}{
 \exseq{E_0[n,q](-L)}{h^*(F^a)^*(Q_h)}{\O_\cA(qL)}.
 }
 Dualizing we get exact
 \eqspl{modif-q-dual}{
 \exseq{p_h^*(F^a)^*(\check Q_h)}{\check E_0[n,q](L)}{\O_\cA((1-q)L+h)}
 }
 The bundles $E_0[n,p^a]$ can also be constructed by induction on $a$.
 In view of the surjection $F^*\O_{\cA}=\O_{p\cA}\to\O_\cA$ with kernel $\O_{(p-1)\cA}(-L-h)$, 
 we have an exact sequence
 \eqspl{q-case-rel}{
 \exseq{F^*E_0[n, p^a]((-p+1)L)}{E_0[n, p^{a+1}]}{\O_{(p-1)\cA}(h)}.
 }
 
 Consequently
 \eqspl{chern-q-case}{
 c(E_0[n, p^{a+1}])=F^*c(E_0[n,p^a](-p+1)L)\frac{1+h}{1+(-p+1)L}.
  }
 In particular,
 \eqspl{c1-q-case}{
 c_1(E_0(-L))=(q-1)h-L, \qquad c_1(E_0)=(q-1)h+(n-1)L.
 } 

\subsection{Construction: $q,k$ case}\label{construction-qk} 

Still more generally, let $1\leq k\leq q$. Starting from the surjection
 \[p_2^*(F^a)^*Q_h\to (F^a)^*(\O_\cA(L))=\O_{q\cA}(qL)\]
 and composing with the restriction map $\O_{q\cA}(qL)\to\O_{k\cA}(qL)$
 we get a surjection
 \[p_2^*(F^a)^*Q_h\to \O_{k\cA}(qL)\]
 The kernel is a rank-$n$ vector bundle that we denote  by $E_0[n,q,k](-L)$.
 Thus we have exact
 \eqspl{nqk-case}{
 \exseq{E_0[n,q,k](-L)}{p_2^*(F^a)^*Q_h}{ \O_{k\cA}(qL)}
 }
 with dual
 \eqspl{nqk-dual}{
 \exseq{p_2^*(F^a)^*\check Q_h}{\check E_0[n,q,k](L)}{ \O_{k\cA}((-q+k)L+kh)}.
 }
 By replacing the middle and right terms in \eqref{nqk-case}, 
 by the complexes $\O(-qh)\to (n+1)\O$ and $\O((q-k)L-kh)\to \O(qL)$ respectively,
 we can describe $E_0[n,q,k](-L)$
 as the cohomology of the following monad:
 \eqspl{nqk-monad}{
 \O(-qh)\xrightarrow{A} (n+1)\O\oplus \O((q-k)L-kh)\xrightarrow{B} \O(qL).
 } Using homogeneous coordinates
 $a_0:\dotsc:a_n$, $x_0:\dotsc:x_n$, the maps are given by
 \begin{equation*}
A=\begin{pmatrix}
a_0^q \\
\vdots \\
a_n^q \\
(a_0x_0+\dotsc +a_nx_n)^{q-k}
\end{pmatrix},\quad
B=\begin{pmatrix}
x_0^q & \hdots & x_n^q & -(a_0x_0+\dotsc +a_nx_n)^{k}
\end{pmatrix}
\end{equation*}
 Similarly $\check E_0[n,q,k][L]$ is the cohomology of the dual monad
 \eqspl{nqk-dual-monad}{
 \O(-qL)\to (n+1)\O\oplus \O(-(q-k)L+kh)\to \O(qh).
 }
 For different $k$, these bundles are related by
 \eqspl{change-k-eq}{
 \exseq{E_0[n,q,k](-L)}{E_0[n,q,k-1](-L)}{\O_\cA((q-k+1)L-(k-1)h)}
 }
 \begin{remarks}\label{remarks5}
  (i) The bundles $E_0[n,q,k]$  for various $n$ are mutually compatible in the sense that
 \eqspl{compatibility}{
 E_0[n,q,k]|_{\P^{n-1}_L\times\P^{n-1}_h}=E_0[n-1, q,k]\oplus\O.
 } 
 This follows from the fact that 
 \[Q_{\P^n}|_{\P^{n-1}}=Q_{\P^{n-1}}\oplus\O.\]
 In view of the nature of the cohomology ring of $ \P^n\times\P^n$, this implies
 \[c_i(E_0[n,q,k])=c_i(E_0[n-1,q,k]), i<n, n\geq 3,\]
 \[c_n(E_0[n,q,k])=a_nL^n+b_nh^n.\]
 (ii) In the extreme case $k=q$, $E_0[n,q,q](-L)$ is just $(F^a)^*$ of the kernel $K$
 of the natural surjection $p_h^*Q_h\to\O_\cA(L)$. But note that $p_{L*}K$ is the kernel 
 of the map $(n+1)\O\to\O(L)$, i.e. $Q_L^*$ and by comparing $c_1$
 it follows that  the natural map $p_L^*p_{L*}K\to K$ is an isomorphism. Thus
 \[E_0[n,q,q](-L)=(F^a)^*p_L^*(Q_L(-L))=\big((F^a)^*p_L^*(Q_L)\big)(-qL).\]
 (iii)  The same divisor $\cA$ can also be realized as $\P(Q_L)$, thus leading to an 
 \emph{a priori} different bundle
 as the kernel of  $p_L^*Q_L\to\O_\cA(ph)$. However we shall see below in \S\ref{A-symmetry} that the latter kernel is
 just the flip of $E_0(-L)$, i.e. pullback under factor interchange.
 \end{remarks}
 \subsection{Restriction on $\cA$}
 Note that 
 \[c_1(E_0|_{\P^n_L\times x})=(n-1)L.\]
 hence
 \[E_0|_{\P^n_L\times x}\simeq (\wedge^{n-1}E_0^*|_{\P^n_L\times x})(n-1).\]
 Restricting \eqref{modif-q} on $\cA$ and using ${\mathrm{Tor}}_1(\O_\cA, \O_\cA)=\O_\cA(-L-h)$ we get exact
 \[0\to\O_\cA((q-1)L-h)\to E_0(-L)[n,q]|_\cA\to p_2^*(F^a)^*(Q_h)|_\cA\to\O_\cA(qL)\to 0.\]
 Note that the right map is part of  $(F^a)^*$ of the relative Euler sequence on $\cA$, hence we get an exact sequence on
 $\cA$:
 \eqspl{e-on-A}{
 \exseq{\O_\cA((q-1)L-h)}{E_0(-L)[n,q]|_\cA}{((F^a)^*\Omega_{\cA/\P^n_h})(qL)}
 }
\par

 Now the pullback of the Euler sequence on $\cA$ is
 \[\exseq{((F^a)^*\Omega_{\cA/\P^n_h})(qL)}{(F^a)^*Q_h\otimes\O_{\cA}}{\O_\cA(qL)}\]
 and this shows that
 $(F^a)^*(\wedge^2Q_h)$ admits ${((F^a)^*\Omega_{\cA/\P^n_h})(2qL)}$ as quotient,
 hence the latter bundle is globally generated. 
 Combining this with \eqref{e-on-A}, we conclude
 \begin{lem}
 $E_0[n,q]((q-1)L+h)|_\cA$ is globally generated.
 \end{lem}
 In fact, the sequence \eqref{e-on-A} splits. More generally,
 for any $k\in[1,q]$, we see in the monad \eqref{nqk-monad} that the map $\O((q-k)L-kh)\to\O(qL)$
 vanishes precisely on $k\cA$, hence we can write, for some locally free $\O_{k\cA}$ module $E'$
 of rank $n-1$
 \[E_0[n,q,k](-L)|_{k\cA}=\O_{k\cA}((q-k)L-kh)\oplus E'.\]
where $E'$ is the cohomology of a monad
\eqspl{monad-e'}{
\O_{k\cA}(-qh)\to(n+1)\O_{k\cA}\to \O_{k\cA}(qL).
}Of course a similar monad also computes $E'|_\cA$.
In particular, for $k=1$ this shows that the sequence \eqref{e-on-A} splits.

 \subsection{Nondegeneracy}\label{nondegeneracy}
 \begin{cor}
 (i) Assume $q\not\equiv 1 \mod n$. Then no twist of $E$ is a pullback of a bundle from a factor $\P^n$.\par
 (ii) Assume $q=2$. Then no twist of  $E$ is  a pullback of a bundle via 
 any map $\P^n\times\P^n\to\P^m$. 
 \end{cor}
 \begin{proof}
 (i) By \eqref{c1-q-case}, there is no twist of $E$ whose $c_1$ is a multiple of $L$ or $h$.\par
 (ii) It suffices to prove this for $n=2$. 
 Then we have
 \[c(E_0(ah+bL))=1+(2a+1)h+(2b+1)L+(2+a+a^2)h^2+(1+2ab+a+b)hL+(2+b+b^2)L^2.\]
 Suppose there is a map $\phi:\P^2_h\times\P^2_L\to\P_\l^m$ with  $\phi^*(\l)=\alpha h+\beta L$
 and a bundle $F$ on $\P^m$ with $c(F)=1+u_1\l+u_2\l^2$. We may assume $u_1\in\{0,1\}$. Then 
 equating $c_1$ we see that $u_1=0$ is impossible. For $u_1=1$ we get
 \[\alpha=2a+1, \beta=2b+1\]
 and then
 \[u_2(\alpha h+\beta L)^2=(2+a+a^2)h^2+2ab+a+b)Lh+(2+b+b^2)L^2.\]
 Then equating coefficients of $h^2$ we get
 \[u_2(2a+1)^2=a^2+a+1\]
 which is evidently possible only for $u_2=1, a\in\{0,1\}$. Similarly $b\in\{0,1\}$.
 Then equating coefficients of $Lh$ yields a contradiction.\par
 Note that in place of $\P^m$ we could have any variety with Chow groups $\Z$
 in degrees 1, 2.
 \end{proof}
%
%

 \subsection{Splitting type along $L$ lines}\label{jumping-sec}
 
 The restriction of $E_0[n,q](-L)$ on a fibre $\P^n_L\times x$ fits in an exact sequence
 \[\exseq{E_0(-L)[n,q]|_{\P^n_L\times x}}{n\O}{\O_A(qL)}\]
 where $A=\cA\cap(\P^n_L\times x)$ is a hyperplane  in $\P^n_L\times x$ . 
 Now consider  restriction on lines $f\subset\P^n_L\times x$  (such lines are called $L$ lines).
 Assume first $f\not\subset\cA$ 
 and let $p=f.A$.  Then we get exact
 \[\exseq{E_0(-L)[n,q]|_f}{n\O_f}{\O_p}.\]
 Therefore clearly 
 \[E_0[n,q](-L)|_f=(n-1)\O_f\oplus \O_f(-1).\]
 Now consider a line $f\subset A$. Then we get exact
 \[0\to{\mathrm{Tor_1}}(\O_A(qL), \O_f)\to E_0[n,q](-L)|_f\to n\O_f\to\O_f(qL)\to 0.\]
 Now the $\mathrm{Tor}$ in question is obviously locally free, of rank $1$,
 hence calculating determinants of all sheaves in this sequence immediately shows
that it is isomorphic to $\O_f(q-1)$.
 The right map is given by the $q$-th powers, hence its kernel is $(F^a)^*Q_L^*|_f=(n-2)\O_f\oplus\O_f(-q)$.
 Therefore we have exact
 \[\exseq{\O_f(q-1)}{E_0(-L)|_f}{(n-2)\O_f\oplus\O_f(-q)}.\]
 This extension obviously splits, so we have
 \eqspl{e-on-f}{ E_0[n,q](-L)|_f=\O_f(q-1)\oplus(n-2)\O_f\oplus\O_f(-q).
 }
 \begin{rem}\label{jumping-rem}
 Note that it follows from \eqref{compatibility} that the jumping lines for $E_0[n-1,q]$
 are precisely the jumping lines for $E_0[n,q]$ contained in $\P^{n-1}_L\times x$.
 \end{rem}
 \subsection{Indecomposability}\label{indecomposability}
 \begin{cor}
 The  bundles $E_0[n,q, k]$ are indecomposable.
 \end{cor}
 \begin{proof} First for $n=2, k=1$ this simply follows
 from the fact that $E_0[2,q,1]$ has jumping lines, e.g. $\cA\cap (\P^2_L\times x)$. 
  For general $n$ and $k=1$ we use induction on $n$.
  By \eqref{compatibility} and induction, if $E_0[n, q, 1]$ is decomposable it must split off an $\O$ factor,
  hence $H^0(E_0[n, q, 1])\neq 0$. But the map $H^0(p_2^*F^*Q_h)=F^*V\to H^0(\O_\cA(pL))=\sym^pV$ where $V=H^0(\O(L))$ is
  clearly injective, hence $H^0(E_0[n, q, 1])=0$.\par
  Finally for  $k>1$ we can use induction on $k$: by \eqref{change-k-eq}, if $E_0[n,q,k]$ were decomposable
  then so too would be its extension $E_0[n,q,k-1]$.
 \end{proof}
 \section{Cohomology}
 Here we fix $n,q$ and  assume that $E_0(-L)=E_0[n,q](-L)$ is given by the elementary modification 
 \eqref{modif-q}, hence the dual $\check E_0$ 
 sits in an exact sequence
 \eqspl{modif-for-e*}{\exseq{(p_h^*F^{a*}\check Q_h)(-L)}{\check E_0}{\O_{\cA}(-qL+h)}.}
 Twisting, we get
 \[\exseq{(p_h^*F^{a*}\check Q_h)((q-1)L+qh)}{\check E_0(qL+qh)}{\O_{\cA}((q+1)h)}.\]
 Note $((F^a)^*\check Q_h)(qh)=\wedge^{n-1}(F^a)^*Q_h$ is globally generated, as is $\O_\cA((q+1)h)$.
 Now the exact sequence
 \eqspl{wedge}{\exseq{(n+1)\O(-qh)}{\frac{(n+1)n}{2}\O}{\wedge^{n-1}(F^a)^*Q_h}}
 shows that $H^1( (F^{a*}\check Q_h)((q-1)L+qh))=0$ if $n>2$ or $q\leq n$, 
 hence the map $H^0(\check E_0(qL+qh))\to H^0(\O_{\cA}((q+1)h))$ is surjective, hence
 $\check E_0(qL+qh)$ is globally generated and has $H^1=0$ if $n>2$ or 
 if $q\leq n\leq 2$.
 \par
 Furthermore, \eqref{wedge} also shows that $H^i((\wedge^{n-1}(F^a)^*Q_h)(th)=0$
 unless $i=0$ or $i=n-1$ and $t<q-n$. Hence we conclude
 \begin{lem}
 For $i>0$ we have \[H^i(\check E_0[n,q]((q+s)L+(q+t)h))=0\] 
 unless  $i=n-1$ and either  $s<q-n$ or $t\leq -q-n$.
 Moreover   $\check E_0[n,q](sh+(t-1)L)$ is globally generated fora all  $s\geq q, t\geq q$.
 \end{lem}
  Note that $\check Q_h=(\wedge^{n-1}Q_h)(-h)$ hence
  \[p_2^*(F^{a*})(\check Q_h)=(p_2^*(F^{a*}\wedge^{n-1}Q_h))(-qh)\]
  and
  \[p_2^*(F^{a*})(Q_h)=(p_2^*(F^{a*}\wedge^{n-1}\check Q_h))(-qh)\] 
  
  Hence it follows from \eqref{modif-q-dual} that
  \eqspl{vanishing-from-modif}{
  H^i(\check E_0[n,q](sh+tL))=0, i>0, {\textrm{ provided}}\ \   s\geq q-n\ \  {\textrm{or}}\ \   t\geq q-n.
  }
  Moreover
  \eqspl{gg-from-modif}{
  \check E_0[n,q](sh+(t-1)L) \ \ {\textrm{is globally generated}}\ \ \forall  s\geq q, t\geq q.
  }
  As for $E_0[n,q]$, note that
  \[E_0[n,q]=(\wedge^{n-1}\check E_0[n,q])((n-1)L+(q-1)h).\]
  and by taking $\wedge^{n-1}$ of \eqref{modif-for-e*} and twisting, we have exact
  \eqspl{wedge-for e}{
 \exseq{(p_2^* (F^{a*}Q_h)(-2(n-1)L-h)}{E_0[n,q]}{\O_{(n-1)\cA}(-(n-1)(q-1)L-(n+q-2)h)}.
  }
  Consequently we have
  \begin{lem}
$ E_0[n,q](sL+th)$  is globally generated and has vanishing $H^i$, 
$ \forall i>0, s\geq (n-1)(q-1), t\geq n+q-2$.
 \end{lem}
  
 Next we study restrictions of our bundles on a fibre $\P^n_L\times x$. We have exact
 \[\exseq{E_0(s-1)|_{\P^n\times x}}{n\O(s)}{\O_A(q+s)}, \quad A=\cA\cap(\P^n\times x)\]
 with dual
 \[\exseq{n\O(s)}{\check E_0(s+1)}{\O_A(-q+s+1)}\]
 which implies that
 \[H^i(\check E_0(t))=0\ \  \forall t\geq q-n, i>0\]
 and moreover $\check E_0(t)$ is globally generated for $t\geq q-1$.
Now taking $\wedge^{n-1}$ of the exact sequence
\[\exseq{n\O(q-1)}{\check E_0(q)}{\O_A}\]
leads to
\[\exseq{n\O((n-1)(q-1))}{(\wedge^{n-1}\check E_0)((n-1)q)}{\O_{(n-1)A}}.\]
Because $E_0=(\wedge^{n-1}\check E_0)(n-1)$, we conclude
\eqspl{e-vanishing-on-fibre}{
H^i(E_0(s))=0, \forall i>0, s\geq(n-1)(q+1).
} 
 
   \section{A-Symmetry}\label{A-symmetry}
The result of this section is used in \cite{p2p2-II} but not in the rest of this paper. 
  Here $E$ is given by the modification sequence over $\P^n_h$, where $q=p^a$, 
  \[\exseq{\check E_0(h)}{p_h^*F^{a*}Q_h}{\O_{\cA}(qL),}\]
  where $\check E_0=E_0(-L)$. In the applications, $q=2$.
  \begin{prop} Assume $q=p=2$. Then
  we have exact
  \eqspl{modif-over-L}{\exseq{\iota^*(\check E_0(h))}{p_L^*F^{a*}Q_L}{\O_{\cA}(qh)}
  }
  or equivalently
  \[\exseq{\check E_0(h)}{p_h^*F^{a*}Q_h}{\O_{\iota^*\cA}(qL)}.\]
  where again the right map is the canonical one with image the set of $q$th powers
  and $\iota$ is the involution exchanging factors (note $\iota^*(\check E_0(h))=(\iota^*\check E_0)(L)=\iota^*E_0$).
  \end{prop}
  In particular, whenever $E_0\simeq\iota^*E_0$ (e.g. when $n=q=2$ and 
   $E_0$ comes from a symmetric monad,
see \S \ref{construction} below) ,
  then $\cA$ is symmetric.\par
  See Corollary  14 of \cite{p2p2-II}  for another, more explicit proof of this result 
  in the special case $n=q=2$. The symmetry will play a critical role in our analysis of
  zero sets in that case.
  \begin{proof}
  Note that because $\cA$ restricts to the analogous divisor on $\P^2_L\times\P^2_h$, 
  the essential case is 
  $n=2$.\par
  By the discussion in \S\ref{jumping-sec}, 
  $\cA$ is the union of the $L$- jumping lines. More precisely, for any line $\l$
  contained in some fibre  $\P^n_L\times x$, we have
  \eqspl{lines-cases}{
  E_0(-L)|_\l=\begin{cases}
  (n-1)\O_\l\oplus\O_\l(-1), \l\not\subset \cA\\
  \O_\l(q-1)\oplus(n-2)\O_\l\oplus\O_\l(-q), \l\subset\cA.
  \end{cases}
  }
  Note that $\cA$ fibres over both $\P^n_h$ and $\P^n_L$ with fibres $\P^{n-1}$
  and to prove \eqref{modif-over-L} it will suffice to
  show the analogous result  for each if its $h$-lines, i.e. lines  $f$ contained in
  fibres over $\P^n_L$.
  Fixing $x\in\P^n_L$ and restricting  the modification sequence 
   on $x\times\P^n_h$, and setting
 $A=(x\times \P^n_h)\cap\cA$, we  get exact
  \eqspl{mod-on-p2}{\exseq{\check E_0|_{x\times\P^n_h}}{F^{a*}Q_h}{\O_{A}}.}
  Now for any line $M\subset x\times\P^n_h$, we have
  $Q_h|_M=\O_M(1)
  \oplus(n-2)\O_M$, hence  
  \[F^{a*}Q_h|_M=\O_M(q)\oplus(n-1)\O_M,\] where the $\O(q)$
  summand comes from the 2-dimensional subspace  $W$ of $V=H^0(\O(h))^*$ corresponding to $M$. Now
  if $M\not\subset A$ and $x=M.A$, we get  an exact restriction 
  \[\exseq{\check E_0|_M}{O_M(q)\oplus(n-1)\O_M}{\O_x}.\]
  If $U\subset V$ denotes the 2-dimensional subspace of $V$ corresponding to $f$, the right map above corresponds to
 the quotient map $V\to V/U$. Now  $U$ and $W$ correspond to mutually transverse subspaces of 
 $V=H^0(\O(h))^*$ intersecting in the 1-dimensional subspace corresponding to $x$. 
 Therefore at $p$, $U$ yields a 1-dimensional subspace of the fibre of $Q_h$ at $p$
 that is different from the one corresponding to $W$, viz. the fibre of $\O(2)$.
 Hence the $\O_M(q)$ summand above surjects to $\O_p$, so that
  \[\check E_0|_M=\O_M(q-1)\oplus(n-1)\O_M.\]
  \par
  On the other hand I claim that for $n=2$, $A$ itself is a jumping lines for $G:=\check E_0|_{x\times\P^2_h}$. 
  Indeed the restriction of \eqref{mod-on-p2} on $A$ yields exact
  \[0\to\O_A(-1)\to G|_A\to F^{a*}Q_h|_A\to \O_A\to 0.\]
  In terms of coordinates, we may assume $A$ has equation $x_0$ while the map $\O(-q)\to3\O$
  with cokernel $ F^{a*}Q_h$ has the form $(x_0^q, x_1^q, x_2^q)$ and the map $G\to\O_A$
  is descended from the map $3\O\stackrel{(1,0,0)}{\to} \O_A$. This yields a map $2\O_A\to G|_A$, 
  essertially $(x_1^q, x_2^q)$, whose kernel is $\O_A(-q)$ which ylelds
  an injection $\O_A(q)\to G|_A$ with cokernel $\O_A(-1)$.  The extension
  \[\exseq{\O_A(q)}{G|_A}{\O(A(-1)}\] automatically splits, so $G|_A=\O_A(q)\oplus\O_A(-1)$.
  This $A$ is a jumping line
  and consequently for $n=2$, we have $\cA_L=\cA_h$.
  As noted above, this implies the same conclusion for all $n\geq 2$.\par
%
%
\end{proof}
 		
  \part{The case $p=n=2$}
  Here we work in char. 2 and undertake a detailed study of our bundle $E$ on $\P^2\times\P^2$. Notably, we will develop an equivalent
  construction of $E$ as the cohomology of a very simple monad,
  where we recall again that a monad is a 3-term complex exact off the middle term. The main result is the following.
  \begin{thm}\label{mod-monad thm}
  In case $p=n=2$, the following two sets of
  bundles on $\P^2_L\times\P^2_h$ are equal (up to isomorphism of bundles):\par
  (i) the bundles $E_{\textrm{mod}}(L)$ where $E_{\textrm{mod}}$  fits in an exact sequence 
  as in \eqref{modif-q}
  \[\exseq{E_ {\textrm{mod}}}{(p_h^*F^*Q_h)}{\O_\cA(2L)}\]
  (ii) the bundles $E_{\textrm{monad}}$ that are the cohomology of a monad
  \[\O\to Q_L\otimes Q_h\to\O(L+h).\]
  \end{thm}

\section{Construction via a monad}\label{construction}
$\P^2_L$ denotes  a copy of $\P^2$ with $L$ a line, $V_L=H^0(\O(L))^*$ and
$Q_L=V_L\otimes\O/\O(-L)$ the universal quotient bundle. Ditto to
$\P^2_h, h, V_h, Q_h$. 
\subsection{Construction}
Set 
\[V_{L, h}=V_L\otimes V_h=H^0(\Hom(Q_L\otimes Q_h, \O(L+h))),\]
a 9-dimensional vector space. 
We have a {symmetric} pairing
\[V_{L, h}\times V_{L, h}\to H^0(\O(L+h))\otimes\det(V_{L, h}),\]
\[(v_1\otimes w_1, v_2\otimes w_2)\mapsto \carets{v_1, v_2}\otimes\carets{w_1, w_2},\]
which corresponds up to scalars to
\[(\phi, \psi)\mapsto \psi\circ \phi.\]
This pairing can be viewed as a 9-dimensional system of symmetric bilinear forms 
whose corresponding system of quadratic forms  is just the  9-dimensional system of quadrics on
$\P^8$ cutting out the Segre image $\P^2_L\times\P^2_h$  
(though  this will not be important in the sequel). Note that
\[\carets{v\otimes w, v\otimes w}=0\]
Therefore in char. 2, we have
\[\carets{\phi, \phi^t}=0 \ \  \forall \phi \in V_{L,h}.\]
Moreover for general $\phi$ (specifically of rank 3), $\phi$ is fibrewise injective while
$\phi^t$ is fibrewise surjective. Hence such a $\phi$ defines a monad
\eqspl{monad-e0}{
\O\stackrel{\phi}{\to} Q_L\otimes Q_h\stackrel{\phi^t}{\to}\O(L+h).
}
We call such a monad \emph{symmetric}.
Then set
\eqspl{E-from-monad}{E_0=\ker(\phi^t)/\im(\phi)\\
E=E_0(L+h).} 
These are  rank-2 vector bundles on $\P^2_L\times\P^2_h$.
As we will show in \S \ref{cohomology},  $\ker(H^0(\phi^t))=\im(H^0(\phi))$, and it follows that
\eqspl{e0-vanish}{h^0(E_0)=0.
}
and likewise
\eqspl{h1-for-e0}{
h^1(E_0)=1.
}
A straightforward calculation shows that $E_0$ has Chern class
\eqspl{ce0}{
c(E_0)=1+L+h+2L^2+2h^2+Lh.
}
The display of the monad \eqref{monad-e0} is as follows:
\small
\eqspl{display-2-2}{
\begin{matrix}
&&&0&&0&\\
&&&\downarrow&&\downarrow&\\
0\to&\O&\to& G_0&\to& E_0&\to 0\\
&\parallel &&\downarrow&&\downarrow&\\
0\to&\O&\to&Q_L\otimes Q_h&\to&F_0&\to 0\\
&&&\downarrow&&\downarrow&\\
&&&\O(L+h)&=&\O(L+h)&\\
&&&\downarrow&&\downarrow\\
&&&0&&0&
\end{matrix}
}
\normalsize
{\bf{ Important note:} } until further notice (that is, until \S \ref{image}), 
$E_0$ denotes the bundle constructed from the monad above 
and likewise for $E$, and not those constructed via a modification as in Part 1. Of course,
our purpose is to \emph{ eventually} prove that they are the same but this is not known at this stage.
 \subsection{A vanishing result}
 As a warmup we will show that for $E_0$ given by the monad \eqref{monad-e0}, 
 we have
  \eqspl{h1=0}{H^1(E_0(h))=0.}  
More complete cohomological results are given in the next section.

\begin{proof}
Note that the claimed vanishing is equivalent to $H^0(E_0(h))=3$.
Via the defining monad, the
 claim is equivalent to surjectivity of the map induced by $\psi$
 \[H^0(Q_L\otimes T_h)\to H^0(\O(L+2h)),\]
 We will prove the surjectivity under the assumption that $\psi$ corresponds to a rank-2 tensor
 \[v_{p_1}\otimes v_{q_1}+v_{p_2}\otimes v_{q_2}, p_1\neq p_2\in \P^2_L, q_1\neq q_2\in P^2_h.\]
 This will  imply the result for a rank-3 tensor as well, which is the case we need.

 Now for $q\in P^2_L, q\in P^2_h, g\in H^0(\O(h))$, we have
 \[\psi(v_p\otimes gv_q)=\carets{v_{p_1}, v_p}\otimes \carets{v_{q_1}\otimes gv_q}+\carets{v_{p_2}, v_p}\otimes \carets{v_{q_2}\otimes gv_q} \]
 Taking $p=p_2$ (resp. $p=p_1$)
 the images generate $\carets{p_1, p_2}\otimes H^0(\O(2h-q_1))$ (resp. $\carets{p_1, p_2}\otimes H^0(\O(2h-q_2))$).
 and together these generate $\carets{p_1, p_2}\otimes H^0(\O(2h))$. Similarly, taking $q=q_1, q=q_2$
 the images generate $H^0(\O(L))\otimes H^0(\O(2h-\carets{q_1, q_2}))$.
 \end{proof}

\section{Cohomology via monad}\label{cohomology}

This part begins our analysis of the bundle $E=E_0(L+h)$ where $E_0$ is as in \eqref{E-from-monad}. 
Key results are:
\begin{enumerate}
\item determination of the cohomology groups $h^iE(aL+bh)$ for $a,b\ge -5$ (Theorem \ref{coh_groups})
\item there is a unique jumping line of $E$ in each fibre $x\times \bP^2$, and the union of these
lines forms a divisor $A\in \vert L+h\vert$ (Corollary \ref{cor111})
\end{enumerate}
\begin{rem}\label{rmk5} If the bundle $E$ is defined by $\phi=\sum v_i\otimes w_i$, then $A$ has the equation
$\sum v_iw_i=0$.
\end{rem} 

\bigskip

In this section we set $\O(aL+bh)=\O(a,b)$ and let $E$ denote $E_0(L+h)$ where $E_0$ is given by the monad \eqref{monad-e0}. 
We shall determine the cohomology of $E(a,b)$ in the
range $a\ge -5$, $b\ge -5$. For $-5\le a,b\le 3$ the result can be seen in 
the tables below. There is no table for $h^4E(a,b)$, because these groups
 are $0$ in the whole range, and blanks correspond to zeros.
\begin{thm}\label{coh_groups}
(i) The higher cohomology groups of $E(a,b)$ vanish for $a\ge -1,b\ge 0$
and for $a\ge 0, b\ge -1$, hence
\begin{equation*}
h^0 E(a,b)=\chi E(a,b)= \frac{a'b'(a'b'-1)}2 -a'^2-b'^2+1 \qquad \text{(where $a'=a+3$, $b'=b+3$)}
\end{equation*}
as long as $a,b\ge -1$ and $(a,b)\ne (-1,-1)$.

(ii) The dimensions of other cohomology groups are given in the following tables, in which blanks
correspond to zeros.\par \bigskip
\bigskip 
\bigskip
{\small
$h^0 E(a,b)$  \hskip 8cm $h^1 E(a,b)$
\par\noindent
\begin{tabular}{c r | *{9}r }
&&&&&$b$ \\
&& -5 & -4 & -3 & -2 & -1 & 0 & 1 & 2 & 3 \\
\hline
&-5 &  \\
&-4 &  \\
&-3 &  \\
&-2 & \\
&-1 &&& && & 3& 9 & 17 & 27\\
$a$ &0&&&&&3 &19&42& 72 & 109\\
&1 &&&&&9&42&89&150 & 225\\
&2 &&&& & 17 & 72 & 150 & 251 & 375 \\
&3 &&&&&27&109&225&375 & 559\\
&4 &&&&&39&153&314&522 & 777 \\
\end{tabular}
\quad
%
%
\begin{tabular}{c r | *{9}r }
&&&&&&$b$ \\
&& -5 & -4 & -3 & -2 & -1 & 0 & 1 & 2 & 3 \\
\hline
&-5 &&&&& &  &  &  &  \\
&-4 &&&&&  1 & 3  & 6 & 10 & 15 \\
&-3 &&&&& 3 & 8 & 15 & 24 & 35 \\
&-2 &&& & 1&3&6&10&15 & 21\\
$a$ &-1 && 1 & 3 & 3& 1 &&&&\\
 &0 &&  3 & 8 & 6&&&&  & \\
&1 && 6 & 15 &10&& && & \\
&2 && 10 & 24& 15 && & && \\
&3 && 15 &35&21&&&&\\
&4 && 21 &48&28&&&& &  \\
\end{tabular}

\bigskip

\noindent
$h^2 E(a,b)$ \hskip 8cm $h^3 E(a,b)$\medskip
\par\noindent
\begin{tabular}{c r | *{9}r }
&&&&&&$b$ \\
&& -5 & -4 & -3 & -2 & -1 & 0 & 1 & 2 & 3 \\
\hline
&-5 &&&&&3& 9 & 17 & 27 & 39  \\
&-4 &&&&&  &  & & &  \\
&-3 &  && 1 & &  &  & & &  \\
&-2 &&& & &&&& & \\
$a$ &-1 & 3& & & &  &&&&\\
 &0 & 9&  & & &&&&  & \\
&1 &17&  & &&& && & \\
&2 &27&  & &  && & && \\
&3 &39&  &&&&&&\\
&4 &53&  &&&&&& &  \\
\end{tabular}
\quad
\begin{tabular}{c r | *{8}r }
&&&&&&$b$ \\
&& -5 & -4 & -3 & -2 & -1 & 0 & 1   \\
\hline
&-5 &1&3&3& 1& & &  \\
&-4 &3 &1 &&&  &  &   \\
&-3 &3 &&  & &  &  &   \\
&-2 &1&& & &&&  \\
$a$ &-1 && & & &  &&\\
 &0 &&  & & &&&   \\
&1 &&  & &&& &  \\
&2 &&  & &  &&  & \\
&3 &&  &&&&\\
&4 &&  &&&& &  \\
\end{tabular}
}
\end{thm}
\bigskip
\begin{proof}
Our general strategy is to apply the Leray spectral sequence with respect to the 
projection to one of the factors. It turns out that only one of the direct image
sheaves is nonzero for each twist.
The cohomology of $E(a,b)$ for $-5\le a\le -1$ can be read off
from the results in Proposition \ref{p4} below.

Since $h^1 (E(0,-1))=h^2(E (0,-2))=h^3 (E(0,-3))=h^4 (E(0,-4))=0$, we conclude
that $h^1 (E(0,l))=0$ for any $l\ge 0$, and inductively that $h^i(E(a,b))= 0$
for any $i>0$, $a,b\ge 0$.
\end{proof}

We now proceed to outline the details of the calculation. First we recall the 
following well-known result.

\begin{lem}\label{l1}
Let $E$ be a vector bundle on $\P^2$, and assume that
$h^1(E(-1))=1$ and $h^2(E(-2))=0$. Then $h^1(E)=0$.
\end{lem}
\begin{proof}
Let $M$ be an arbitrary line. $E$ splits on $M$ as $\oplus\cO_M(a_i)$,
and the assumptions imply that $h^1 ((E(-1)\otimes\cO_M))\le 1$.
Thus $a_i\ge -1$ for all $i$ and $h^1( (E\otimes\cO_M))= 0$.

Now assume that $h^1(E)>0$. Since $h^1(E\otimes\cO_M)=0$, the multiplication by the linear form
corresponding to $M$ in $H^1(E(-1))\to H^1(E)$ must be surjective, hence $h^1(E)=1$,
and the multiplication map is even bijective. On the other hand the multiplication map
\begin{equation*}
H^1 ((-1)) \times H^0(\cO_{\bP^2}(1)) \to H^1(E)
\end{equation*}
cannot be injective, as the source has dimension $3$. Therefore there must be a linear form
corresponding to a line $M$ so that the induced multiplication is zero.

We have arrived at a contradiction, and therefore we must have $h^1(E)=0$.
\end{proof}
\begin{prop}\label{p2}
Let $E_x$ be the restriction of $E$ to a fiber $\bP^2_L\times\{x\}$.
\begin{enumerate}
\item[\textup{1.}]
The cohomology $h^i (E_x(l))$ in the range $-5\le l\le -1$ is as follows:
\bigskip

\emph{
\begin{tabular}{c c | *{6}c }
& 2 & 2 &&&&  \\
$i$ &1 & & 1 & 2 & 1 &  \\
& 0 &&&&&2 \\
\hline
&& -5 & -4 & -3 & -2 & -1 \\
&&&& $l$
\end{tabular}
}
\item[\textup{2.}]
The cokernel of the map $H^0 (E_x(-1)) \otimes\cO_{\bP^2} \to E_x(-1)$ is
isomorphic to $\cO_M(-1)$ for a line $M$.
\item[\textup{3.}] $E_x(-1)$ splits generically as $\cO\oplus\cO(1)$. $M$ is its only
jumping line, and $E_x(-1)\otimes\cO_M\cong\cO_M(2)\oplus\cO_M(-1)$.
\end{enumerate}
\end{prop}
\begin{proof}
1. Recall that the bundle $E$ is given by the monad
\begin{equation*}
0 \to \cO(1,1) \to Q_L(L)\otimes Q_h(h) \to \cO(2,2)\to 0
\end{equation*}
Hence the restriction of $E$ to a fiber $\bP^2\times\{x\}$ is given by
the restricted monad
\begin{equation*}
0\to \cO(1) \to 2Q(1) \to \cO(2)\to 0.
\end{equation*}
We note that $c_1(E_x)=3$, $c_2(E_x)=4$ and $\chi\big(E_x(l)\big)= (l+4)(l+2)-1$.

The cohomology for $E_x(-3)$ and $E_x(-2)$ can be read off from
the display of the monad via a diagram chase.

Lemma \ref{l1} implies that $h^1(E(-1))=0$, hence $h^0 (E_x(-1))=\chi (E_x(-1))=2$.

The cohomology for $E_x(-4)$ and $E_x(-5)$ follows by Serre duality.

2. The cokernel of the map $H^0 (E_x(-1)) \otimes\cO_{\bP^2} \to E_x(-1)$ is a
sheaf of projective dimension $1$ supported on a divisor of degree
$c_1 E_x(-1)=1$, hence it is locally free on a line $M$. Its Hilbert polynomial
is $\chi E_x(l-1) -\chi \big(2\cO(l)\big) = l$, hence it has rank $1$ and
is isomorphic to $\cO_M(-1)$.

3. Restricting the sequence $0\to 2\cO\to E_x(-1)\to \cO_M(-1)$ to $M$,
one finds an exact sequence $0\to \cO_M(2) \to E_x(-1)\otimes \cO_M\to \cO_M(-1)$
which necessarily splits.

On any other line the restriction of $E_x(-1)$ is globally generated, hence the splitting
type must be $\cO\oplus \cO(1)$.
\end{proof}
\begin{rem}
Hulek \cite{hulek-bundles-1} investigated stable vector bundles $\cF$ of rank $2$ 
on $\P^2$ with odd first Chern class
in characteristic $0$.
He normalizes them so that $c_1=-1$ and studies ``jumping lines of the second kind'',
i.e., lines $M$ with the property that $H^0(M^2,\cF\otimes \cO_{M^2})\ne 0$, where
$M^2$ is the double structure on $M$ induced by its embedding in $\bP^2$.

Hulek defines a sheaf supported on a divisor $C(\cF)$ of degree $2c_2(\cF)-2$ in $(\bP^2)^\vee$
and proves that this sheaf characterizes $\cF$.
He further constructs moduli spaces $M(-1,c_2)$.

Regarding our case $c_2=2$ (in char. 0), he shows that:
\begin{enumerate}
\item $C(\cF)$ consists of two distinct lines. Their point of intersection corresponds to the unique
jumping line of $\cF$.
\item All bundles with $c_2=2$ are projectively equivalent under automorphisms of $\bP^2$.
\item
The moduli space $M(-1,2)$ is isomorphic to the quotient of $(\bP^2)^\vee\times(\bP^2)^\vee-\Delta$
by the natural symmetry action that exchanges the two factors.
\end{enumerate}
These results do not extend to our example in characteristic $2$: It is not hard to see (using the
alternative construction of $E$ from the following section) that all lines are jumping lines of the second kind for $E_x$, hence the
support of $C(E_x)$ is all of $(\bP^2)^\vee$.

\end{rem}
\begin{prop}\label{p4}
The direct image sheaves $R^i pr_{2,*}E(l,0)$ in the range $-5\le l\le -1$ are as follows:
\bigskip

\emph{
\begin{tabular}{c c | *{6}c }
& 2 & $R(2)$ &&&&  \\
$i$ & 1 &&$\cO_{\bP^2}(1)$ & $Q(1)$ & $\cO_{\bP^2}(2)$ &  \\
& 0 &&&&& $R(1)$ \\
\hline
&& -5 & -4 & -3 & -2 & -1 \\
&&&& $l$
\end{tabular}
}

\noindent where $R=pr_{2,*}E(-1,-1)$ is a rank $2$ vector bundle with Chern class
$c(R)=1+3h^2$, Euler characteristic $\chi R(l)=l^2+3l-1$
and the following cohomology $h^i R(l)$ in the range $-4\le l\le 1$
\bigskip

\emph{
\begin{tabular}{c c | *{6}c }
& 2 & 3 & & & & &  \\
$i$ & 1 && 1 & 3 & 3 & 1 &  \\
& 0 &&& & & & 3 \\
\hline
&&-4 &  -3 & -2 & -1 & 0 & 1 \\
&&&&& $l$
\end{tabular}
}
\end{prop}
\begin{proof}
The calculation of the direct images follows the same approach as for the fibers in the
previous Proposition: The direct images for $l=-3$ and $l=-2$ can be derived from the display
of the monad via a diagram chase. For $l=-1$ we use from Proposition \ref{p2} that the higher direct
images vanish. The columns for $l=-5$ and $l=-4$ follow by relative duality for the projection.

The direct image of the monad for $E(-1,-1)$ produces the monad
\begin{equation*}
0\to \cO \to 3Q \to 3\cO(1)\to 0
\end{equation*}
for $R$, from which we can compute its Chern class. The cohomology for $l=-1$ and $l=0$
can be determined most easily from the display of the dual monad.

Lemma \ref{l1} implies that $h^1(R(1))=0$, hence $h^0(R(1))=\chi R(1)=3$, and the other columns
follow by Serre duality.
\end{proof}

\begin{cor}\label{cor111}
There is a short exact sequence
\begin{equation*}
0\to pr_2^*R\to E(-L-h)\to \cO_\cA(-L+2h) \to 0.
\end{equation*}
where $\cA$ is a divisor in $\vert L+h\vert$.
\end{cor}

\begin{proof}
The cokernel of the map 
\begin{equation*}
pr_2^*pr_{2,*}\big( E(-1,-1)\big) \otimes\cO_{\bP^2} \to E(-1,-1)
\end{equation*}
is a sheaf of projective dimension $1$ supported on a divisor in the class
\begin{equation*}
c_1 E(-1,-1)-c_1 \big(pr_2^*pr_{2,*} E(-1,-1)\big)=L+h,
\end{equation*}
hence it is locally free on a divisor $\cA$ in $\vert L+h\vert$.

Our analysis of the restriction to the fibres implies that the cokernel
has rank $1$ and is isomorphic to $\cO_\cA(-L+ah)$ for some integer $a$.

Finally, a short calculation yields $a=2$.
\end{proof}

\bigskip

 \section{From monad to elementary modification}
\bigskip
In this section, we consider a bundle $E$ defined via a self-dual monad 
\begin{equation*}
\begin{tikzcd}[column sep=small]
0\ar[r] & \cO(L+h) \ar[r,"\phi"] & Q_L\otimes Q_h\ar[r,"{\phi^\vee}"] & \cO(2L+2h) \ar[r] & 0.
\end{tikzcd}
\end{equation*}
We show (Theorem \ref{thm22}) that $E(-2L-h)$ can be represented as an elementary modification,
namely as the kernel of the composition
\begin{equation*}
\begin{tikzcd}[column sep=small]
pr_2^*F^*Q\ar[r] & \pi_2^*F^*Q\ar[r,"F^*\psi"] & \cO_\cA(2L)
\end{tikzcd}
\end{equation*}
where $\cA\cong \P(Q)$ is a smooth divisor in $\vert L+h\vert$, $\pi_2\colon \cA\to\P^2$ the restriction
of the canonical projection $pr_2$, and
$\psi$ is the canonical map $\pi_2^*Q\to\cO_\cA(L)$ of sheaves on $\cA$. This will, in particular, 
prove half of Theorem \ref{mod-monad thm}.
\medskip

This representation has several  other important applications:
\begin{enumerate}
\item
a second monad (Corollary \ref{cor23}),
\item
 the symmetry of $\cA$ (Corollary \ref{symmetry-cor}),
\item
and the splitting of $E$ on $\cA$ (Corollary \ref{cor25}).
\end{enumerate}
These are all specializations
of properties of the general construction in \S 2.

\medskip

Later in \S \ref{image} we will show show that the representation via the elementary modification
induces the monad that we started with.

We assume $E_0$ is the cohomology of a  monad \eqref{monad-e0}.

Recall the exact sequence
 \eqspl{e-quot0}{
 \exseq{p_2^*(R)}{E_0}{\O_{\mathcal A}(-L+2h)}
 }
 from corollary \ref{cor111}.
 
 \begin{prop}
 $\cA$ is isomorphic to the $\bP^1$-bundle $\bP(Q)$
 in its standard embedding given by the global sections
 $3\cO\to Q$.
 \end{prop}
 \begin{proof}
 We know from our investigation of the cohomology of $E$
 that $\cA$ is a divisor in $\vert L+h\vert$, fibered in lines over the second
 factor of $\bP^2$. These lines represent the jumping lines of the
 vector bundle $E$ in the fibers $\bP^2\times\{x\}$.
 
 Applying $pr_{2,*}$ to the sequence $0\to\cO(-h)\to\cO(L)\to\cO_\cA(L)\to 0$
 we find a sequence $0\to \cO(-1)\to 3\cO \to pr_{2,*}\cO_\cA(L)\to 0$ on $\bP^2$
 where the last term is a rank $2$ locally free sheaf on $\bP^2$.
 However, there is only one sheaf with such a resolution, the twisted tangent bundle
 $Q$, hence $\cA\cong \bP(Q)$,
 and $\cO_\cA(L)$ is isomorphic to the twisting sheaf $\cO_{\bP (Q)}(1)$. 
 \end{proof}
 
  \begin{thm}\label{thm22}
  $E(-2L-h)$ is isomorphic to the kernel of the composition
  \begin{equation*}
  \begin{tikzcd}[column sep=small]
  pr_2^*F^*Q\ar[r] & \pi_2^*F^*Q\ar[r,"F^*\psi"] & \cO_\cA(2L)
  \end{tikzcd}
  \end{equation*}
  where $\psi$ is the canonical map $\pi_2^*Q\to\cO_\cA(L)$ of sheaves on $\cA$.
  \end{thm}
  \begin{proof}
  Recall from §\ref{cohomology} that we have a short exact sequence
  \begin{equation*}
  0\to pr_2^*R\to E(-L-h)\to \cO_\cA(-L+2h) \to 0
  \end{equation*}
  where $R=pr_{2,*}E(-L-h)$.
  
  Its twisted dual (by $\cO(h)$) is a sequence
  \begin{equation}\label{eq1}
  0\to E(-2L-h) \to (pr_2^*R)(h)\to\cO_\cA(2L)\to 0
  \end{equation}
  
  Our plan is to identify the two sheaves on the right and the map between them.
  
  Restricting \eqref{eq1} to $\cA$, we obtain $0\to\cO_\cA(-2L+2h)\to \pi_2^*R(1)\to\cO_\cA(2L)\to 0$,
  hence applying $\pi_{2,*}$ yields an exact sequence
  \begin{equation}\label{eq2}
  0 \to R(1) \to \sym^2 Q \to \cO_{\bP^2}(1)\to 0.
  \end{equation}
  
  The morphism $\sym^2Q \to \cO(1)$ corresponds to a section of the twisted divided power
  $(D^2 Q)(-1)$. By Example 1 above, this section is essentially unique, and given by
  the inclusion of the (twisted) exterior power.
  
  As a consequence, the sequence \eqref{eq2} agrees with the sequence in 
  Lemma 1 above,
  and $R(1)\cong F^*Q$, hence $pr_2^*(R(1))\cong pr_2^*F^*Q$.
  
  Finally, consider the diagram
  \begin{equation*}
  \begin{tikzcd}[column sep= small]
  0\ar[r] & \cO_\cA(-2L+2h) \ar[r] \ar[dr, dashrightarrow] & F^* \pi_2^*Q\ar[r] \ar[d,"F^*\psi"] & \cO_\cA(2L) \ar[r] & 0 \\
  && F^*\cO_\cA(L)
  \end{tikzcd}
  \end{equation*}
  where $\psi$ is the Frobenius pullback of the canonical map $\pi_2^*Q_h\to \cO_\cA(L)$.
  
  Since the dashed diagonal map corresponds to a section of $\cO_\cA(4L-2h)$, 
  hence must be $0$, $F^*\psi$ factors through
  the map $F^*\pi_2^*Q \to\cO_\cA(2L)$, and for degree reasons the resulting
  map $\cO_\cA(2L)\to F^*\cO_\cA(L)$ is an isomorphism.
  
  We conclude that the map $\pi_2^*R(1)\to\cO_\cA(2L)$ in the restriction of \eqref{eq1}
  to $\cA$ is isomorphic to the Frobenius pullback $\pi_2^*F^*Q \to F^*\cO_\cA(L)$,
  as required to finish the proof.
  \end{proof}
  
  For the next result, we choose coordinates so that
  $\bP^2\times\bP^2=\Proj[a,b,c,x,y,z]$, and the equation of $\cA$ is
  given by $f=ax+by+cz$.
  
  \begin{cor}\label{cor23}
  $E(-2L-h)$ is isomorphic to the homology of the monad
  \begin{equation*}
  \begin{tikzcd}[column sep=small]
  0\ar[r] & \cO(-2h)\ar[r,"\phi_2"] & \cO^{\oplus 3}\oplus\cO(L-h)\ar[r,"\phi_1"] & \cO(2L)\ar[r] & 0
  \end{tikzcd}
  \end{equation*}
  where the maps are given by $\phi_2=\begin{pmatrix} x^2 & y^2 & z^2 & f \end{pmatrix}^t$
  and $\phi_1=\begin{pmatrix} a^2& b^2 &c^2& f  \end{pmatrix}$
  \end{cor}
  \begin{proof}
  By Theorem \ref{thm22}, $E(-2L-h)$ 
  can be represented as the kernel of the right column in the following diagram:
  \begin{equation*}
  \begin{tikzcd}
  0 \ar[r] & \cO(-2h) \ar[r,"\begin{pmatrix} x^2 \\ y^2 \\ z^2 \end{pmatrix}"] \ar[d,"f"] & 
  \cO^{\oplus 3} \ar[r] \ar[d,"( a^2\enskip b^2\enskip c^2 )"] & pr_2^*F^*Q \ar[r] \ar[d] & 0 \\
  0 \ar[r] & \cO(L-h) \ar[r,"f"] & \cO(2L) \ar[r] & \cO_\cA(2L) \ar[r] & 0
  \end{tikzcd}
  \end{equation*}
  The mapping cone of the two rows, with the right column excluded, provides the required complex.
  \end{proof}
  As a consequence of the above monad, we get another proof of the symmetry result, 
Proposition \ref{A-symmetry} above:  
  \begin{cor}\label{cor24}\label{symmetry-cor}
  Assume $E$ is given by a symmetric monad \eqref{monad-e0} .
  $E(-L-2h)$ is isomorphic to the kernel of the composition
  \begin{equation*}
  \begin{tikzcd}[column sep=small]
  pr_1^*F^*Q\ar[r] & \pi_1^*F^*Q\ar[r,"F^*\psi'"] & \cO_\cA(2h)
  \end{tikzcd}
  \end{equation*}
  where $\psi'$ is the 
  canonical map $\pi_1^*Q\to\cO_\cA(h)$.
  
  In particular, $\cA$ is also the locus of the jumping lines 
  of the restriction of $E$ to the fibers $\{x\}\times\bP^2$.
  \end{cor}
  \begin{proof}
  The dual of the monad in Corollary \ref{cor23} is a monad for
  $E(-L-2h)$. Unwinding it we obtain the diagram
  \begin{equation*}
  \begin{tikzcd}
  0 \ar[r] & \cO(-2L) \ar[r,"\begin{pmatrix} a^2 \\ b^2 \\ c^2 \end{pmatrix}"] \ar[d,"f"] & 
  \cO^{\oplus 3} \ar[r] \ar[d,"( x^2\enskip y^2\enskip z^2 )"] & pr_1^*F^*Q \ar[r] \ar[d] & 0 \\
  0 \ar[r] & \cO(h-L) \ar[r,"f"] & \cO(2h) \ar[r] & \cO_\cA(2h) \ar[r] & 0
  \end{tikzcd}
  \end{equation*}
  where $E(-L-2h)$ is the kernel of the right column. The left column establishes the first statement
  of the corollary.
  
  The second statement follows by restricting the 
  right column to the fibers of $pr_2$.
  \end{proof}
  
  \begin{cor}\label{cor25}
  The restriction of $E$ to $\cA$ splits as $\cO_\cA(3L)\oplus\cO_\cA(3h)$, and the subbundle
  $\cO_\cA(3L)$ $($resp. $\cO_\cA(3h))$ is uniquely determined as the image of the map
  $(pr_1^*F^*Q)(L)\to E$ $($resp. $ (pr_2^*F^*Q)(h)\to E)$.
  In particular, the canonical map 
  \begin{equation*}
  (pr_1^*F^*Q)(h)\oplus (pr_2^*F^*Q)(L)\to E
  \end{equation*}
  is surjective.
  \end{cor}
  \begin{proof}
  We restrict the exact sequence from Corollary \ref{cor111}
  \begin{equation*}
  0\to (pr_2^*F^*Q)(-h)\to E(-L-h)\to \cO_\cA(-L+2h) \to 0.
  \end{equation*}
  to $\cA$ and obtain
  \begin{equation*}
  0\to \cO_\cA(-2L+h) \to
  (\pi_2^*F^*Q)(-h)\to E\otimes\cO_\cA(-L-h)\to \cO_\cA(-L+2h) \to 0.
  \end{equation*}
  Twisting by $\cO(L+h)$ we find an epimorphism $E\otimes\cO_\cA\to\cO_\cA(3h)$
  with kernel $\cO_\cA(3L)$ which is the image of $(\pi_2^*F^*Q)(L)$.
  
  By symmetry (Corollary \ref{cor24}) we obtain another epimorphism $E\otimes\cO_\cA\to\cO_\cA(3L)$
  with kernel $\cO_\cA(3h)$ which is the image of $(\pi_1^*F^*Q)(h)$.
  
  The injection $\cO_\cA(3h)\to E\otimes\cO_\cA$ now provides the required splitting of the
  surjection $E\otimes\cO_\cA\to\cO_\cA(3h)$.
  
  The remaining statements of the corollary follow immediately.
  \end{proof}
  
  \section{From elementary modification to monad}\label{image}
  \subsection{Cohomology} Here we assume $E_0$ is given by a (dual) elementary modification
  \eqspl{e-quot}{
  \exseq{(p_2^*(F^*Q_h))(-h)}{E_0}{\O_{\mathcal A}(-L+2h)}
  }
  where  $\mathcal A$ is a divisor  of class $L+h$
  which is a $\P^1$-bundle over $\P^2_h$ of the form $\P(Q_h)$.\par
  \par
  
   Twisting \eqref{e-quot} by $L+h$ and using $H^1(R(h))=0$, we recover the fact that $E$
  is globally generated, though $E(-L)$  and $E(-h)$ are not. 
  Also this vanishing plus the vanishing of $H^1(\O_{\mathcal A}(3h))$ (obvious)
  imply again that $H^1(E)=0$, hence $h^0(E)=19$. The same sequence also implies easily that
  \eqspl{e-vanishing-final}{H^i(E(aL+bh))=0, \forall i\geq 1, a,b\geq 0.
  }
  Note that it follows from the results of the previous section that the fibres of $\mathcal A$ 
  over the two factors are precisely the
  'jumping rulings' of $E$: i.e. that $E_{\\l\times x}\simeq\O(2)\oplus\O(1)$ when
  $\l\times x$ is not  a ruling of $\mathcal A$ over $\P^2_h$ and $\O(3)\oplus\O$ otherwise; and ditto with fibres
  interchanged.
  \par
  \subsection{The monad}
  Here the characteristic is 2 and $R=(F^*Q_h)(-h)$.
  Finally we claim that the exact sequence \eqref{e-quot} implies that $E$ is the cohomology of a 
  monad \eqref{monad-e0},
  extending the fibrewise result in the previous section. To see this we use the projection $pr_2$ 
  and the corresponding fibrewise Beilinson monad, which shows that $E$ is the abutment of a
  spectral sequence with $E_1$ term
  \[E_1^{p,q}=R^qp_{2*}(E((-3+p)L))\otimes \wedge^{2-p}(Q_L(-L)).\]
  Now  the vanishing of the higher direct images for $q\neq 1$ follows from the corresponding
  fibrewise result, and the evaluation of the images for $q=1$ is done similarly, using that
  $\mathcal A=\P(Q_h)$.

  \section{Appendix 1: The graded module of sections of $E$}
    
   Here we present a resolution of the
    module of sections of $E$ (Theorem \ref{thm32}) and generators for the ideal
    sheaves of the zero schemes (Theorem \ref{thm33}). Some results were obtained
    with the aid of the computer algebra package \emph{Macaulay2} \cite{M2}.
    The supporting code can be found in the next appendix.

    \begin{prop}\label{gens}
    The bigraded module $\oplus_{m,n}H^0(E(m,n))$ has $7$ minimal 
    generators: three each in bidegrees $(-1,0)$, $(0,-1)$ 
    and one in bidegree $(0,0)$.
    \end{prop}
    \begin{proof}
    The computation of the cohomology in section 5 exhibits three
    generators each in bidegrees $(-1,0)$ and $(0,-1)$. 
    Furthermore, the homomorphisms $H^0(E(-L))\otimes H^0(\cO(L))\to H^0(E)$
    (resp. $H^0(E(-h))\otimes H^0(\cO(h))\to H^0(E)$) are both injective, and
    the restrictions of these sections to $\cA$ lie by Corollary \ref{cor25}
    in different subbundles of $E\otimes\cO_\cA$, hence the images
    intersect only in $0$. Therefore their sum generates an $18$-dimensional 
    subspace of $H^0(E)$.
    Since $h^0(E)=19$, one additional generator is needed in bidegree $(0,0)$. 
    These together generate the sections of $E(m,n)$ in all twists $m,n\ge 0$.
    \end{proof}
    \subsection{The minimal bigraded resolution for $E$}
    
    \begin{thm}\label{thm32}
    The vector bundle $E$ has a minimal bigraded resolution
    with the following terms:
    
    {\small
    \begin{tikzpicture}[commutative diagrams/every diagram]
    \node (P) at (-1,0) {$0$};
    \node (P0) at (0.7,0) {$\cO(-2,-2)$};
    \node (P1) at (3,.5) {$\cO(-2,-1)^{\oplus 3}$};
    \node (P2) at (3,-.5) {$\cO(-1,-2)^{\oplus 3}$};
    \node (P3) at (5.2,1) {$\cO(-2,-0)^{\oplus 3}$};
    \node (P4) at (5.2,0) {$\cO(-1,-1)^{\oplus 8}$};
    \node (P5) at (5.2,-1) {$\cO(0,-2)^{\oplus 3}$};
    \node (P6) at (7.5,1.5) {$\cO(-2,1)$};
    \node (P7) at (7.5,.5) {$\cO(-1,0)^{\oplus 6}$};
    \node (P8) at (7.5,-.5) {$\cO(0,-1)^{\oplus 6}$};
    \node (P9) at (7.5,-1.5) {$\cO(1,-2)$};
    \node (P10) at (9.5,1) {$\cO(0,1)^{\oplus 3}$};
    \node (P11) at (9.5,0) {$\quad\cO\quad$};
    \node (P12) at (9.5,-1) {$\cO(1,0)^{\oplus 3}$};
    \node (P13) at (11,0) {$E$};
    \node (P14) at (12,0) {$0.$};
    \node (Q1) at (3,0) {$\quad\oplus\quad$};
    \node (Q2) at (5.2,.5) {$\oplus$};
    \node (Q3) at (5.2,-.5) {$\oplus$};
    \node (Q4) at (7.5,-1) {$\oplus$};
    \node (Q5) at (7.5,0) {$\quad\oplus\quad$};
    \node (Q6) at (7.5,1) {$\oplus$};
    \node (Q7) at (9.5,-.5) {$\oplus$};
    \node (Q8) at (9.5,.5) {$\oplus$};
    \path[commutative diagrams/.cd, every arrow, every label]
    (P) edge node {} (P0)
    (P0) edge node {} (Q1)
    (Q1) edge node {} (P4)
    (P4) edge node {} (Q5)
    (Q5) edge node {} (P11)
    (P11) edge node {} (P13)
    (P13) edge node {} (P14);
    \end{tikzpicture}
    }
    \end{thm}
    \begin{proof}
    Proposition \ref{gens} showed that $\oplus_{m,n} H^0 (E(m,n))$ is
    generated by $H^0 (E(0,-1))$, $H^0 (E(-1,0))$ and one additional section of $H^0 (E)$.
    
    The computer algebra package \emph{Macaulay2} \cite{M2} can create a bigraded module
    and compute its minimal bigraded resolution.
    A priori we only know that the sheafification of the bigraded module will agree with $E$. However,
    inspection of the generators of this module show that their numbers and degrees agree
    with those of $\oplus_{m,n} H^0 (E(m,n))$, hence the resolution captures all sections
    of all twists of $E$.
    \end{proof}
    \subsection{generators of the ideal sheaf of a section of $E$}
  
  In the following, we let $k[a,b,c,x,y,z]$ be the bihomogeneous coordinate ring of $\bP^2\times\bP^2$.
    
    \begin{thm}\label{thm33}
    Let $s\in H^0 (E(m,n))$ be a non-zero section. Then there exist functions
    $f_1,f_2,f_3\in H^0(\cO(m+1,n))$, $g_1,g_2,g_3\in H^0(\cO(m,n+1))$ and $h\in H^0(\cO(m,n))$
    such that the homogeneous ideal of the zero scheme $Z$ of $s$ is generated by the $7$ entries of
    the row vector
    \begin{equation*}
    w=\begin{pmatrix}
    f_1 & f_2 & g_1 & f_3 & g_2 & g_3 & h
    \end{pmatrix}
    \cdot B \end{equation*}
    where $B$ is the following $7\times 7$-matrix:
    
    {\small
    \begin{equation*}
    {
    \left[\begin{matrix}
    
    b^2y^2+c^2z^2 &   axy^2+by^3+cy^2z &a^2y^2 & 0 \\
    b^2x^2  & ax^3+bx^2y+cx^2z &a^2x^2+c^2z^2  &   axz^2+byz^2+cz^3 \\
    ab^2x+b^3y+b^2cz &a^2x^2+b^2y^2  &   a^3x+a^2by+a^2cz& a^2z^2 \\
    c^2x^2 & 0 & c^2y^2  & axy^2+by^3+cy^2z\\
    0  &c^2x^2 & ac^2x+bc^2y+c^3z& b^2y^2+c^2z^2 \\ 
    ac^2x+bc^2y+c^3z& c^2y^2 & 0  &   a^2y^2\\
    b^2cxy+bc^2xz &  acx^2y+bcxy^2 &  a^2cxy+ac^2yz &  aby^2z+acyz^2  
    \end{matrix}\right.
    }
    \end{equation*}
    \begin{equation*}
    \left.\begin{matrix}
    & axz^2+byz^2+cz^3 &a^2z^2 & aby^2z+acyz^2 \\
    &0 & b^2z^2 & abx^2z+bcxz^2 \\
    &b^2z^2 & 0  &  a^2bxz+ab^2yz \\
    & ax^3+bx^2y+cx^2z& a^2x^2+b^2y^2 &  acx^2y+bcxy^2 \\
    &b^2x^2    &      ab^2x+b^3y+b^2cz& b^2cxy+bc^2xz \\
    & a^2x^2+c^2z^2   &  a^3x+a^2by+a^2cz &a^2cxy+ac^2yz \\
    & abx^2z+bcxz^2  & a^2bxz+ab^2yz &  0 
    \end{matrix}\right]
    \end{equation*}
    }
    
    Conversely, the entries of $\begin{pmatrix} f_1 & f_2 & g_1 & f_3 & g_2 & g_3 & h
    \end{pmatrix} \cdot B$
    generate the zero scheme of a section of $E(m,n)$ for any functions $f_1,\dotsc,g_3, h$
    as specified above.
    \end{thm}
    
    The matrix $B$ is presented in the form provided by \emph{Macaulay2}.
    The entries of $w$ are not ordered by bidegree.
    
    \begin{proof}
    (compare \cite[5.1.11]{CLO})
    Let $s$ be a section of $E(a,b)$, and consider the following diagram
    \begin{equation*}
    \begin{tikzcd}[column sep=small]
    && E_1(a,b) \ar[d] \\
    && E_0(a,b) \ar[d] \ar[dr,dashed] \\
    0 \ar[r] & \cO \ar[r,"s"] & E(a,b) \ar[r] & \cO(2a+3,2b+3)\ar[r] & \cO_Z \ar[r] & 0
    \end{tikzcd}
    \end{equation*}
    where $E_1 \to E_0 \to E\to 0$ is the presentation of $E$ from Theorem \ref{thm32} above.
    
    The dashed
    diagonal arrow determines bihomogeneous polynomials $P_1,P_2,P_3\in H^0(\cO(m+2,n+3))$,
    $Q_1,Q_2,Q_3\in H^0(\cO(m+3,n+2))$ and $R\in H^0(\cO(m+3,n+3))$.
    
    Now write $v^t=\begin{pmatrix} P_1 & P_2 & P_3 & Q_1 & Q_2 & Q_3 & R \end{pmatrix}$ 
    and let $A$ be the $7\times 14$-matrix specifying the
    map $E_1\to E_0$. Since $v^t$ corresponds to a section of $E(m,n)$,
    we conclude that $v^t\cdot A=0$.
    
    The latter is equivalent to $A^t\cdot v=0$, $v\in\ker A^t$. Since all rings and modules are
    noetherian, $\ker A^t$ is finitely generated, and a set of generators can be calculated by
    \emph{Macaulay2}, forming the columns of a matrix $B^t$.
    
    Hence we can write $v=B^t\cdot w^t$, where
    $w=\begin{pmatrix} f_1 & f_2 & g_1 & f_3 & g_2 & g_3 & h \end{pmatrix}$.
    Inspection of $B$ shows that the entries of $w$ are homogeneous polynomials
    with the stated bidegrees.
    
    Conversely, given any vector $w$ with entries as above, we have
    $(w\cdot B) \cdot A= w\cdot (B\cdot A)=w\cdot 0=0$, hence $w\cdot B$
    determines a section of $E(m,n)$.
    \end{proof}
    
    \begin{rem}
    We know that the sheaf map
    \begin{equation*}
    F^*Q_L(h)\oplus F^*Q_h(L)\to E
    \end{equation*}
    is surjective. Therefore the ideal sheaf of the zero scheme of any section can be
    generated by the first $6$ entries of $w$ alone. In Theorem \ref{thm33} we showed
    that sometimes even the homogeneous ideal is generated by only these $6$ polynomials.
    \end{rem}
    
    \begin{cor}\label{cor35}
    Let $s$ be a section of $E(m,n)$ corresponding to the row vector
    \[w=\begin{pmatrix}
    f_1 & f_2 & g_1 & f_3 & g_2 & g_3 & h
    \end{pmatrix}\]
    where
    $f_1,f_2,f_3\in H^0(\cO(m+1,n))$, $g_1,g_2,g_3\in H^0(\cO(m,n+1))$ and $h\in H^0(\cO(m,n))$.
    
    The restriction of $s$ to $\cA$ agrees with the section $(p,q)$ of $\cO_\cA(m+3,n)\oplus\cO_\cA(m,n+3)$
    where
    \begin{align*}
    p & = f_1\cdot a^2 + f_2\cdot b^2 + f_3\cdot c^2 + h\cdot abc\\
    q & = g_1\cdot z^2 + g_2\cdot x^2 + g_3\cdot y^2 + h\cdot xyz
    \end{align*}
    \end{cor}
    \begin{proof}
    $p$, $q$ depend $k[a,b,c,x,y,z]$-linearly on the $f_i$, $g_i$, $h$ so that it suffices to check
    the equations on the basis vectors, e.g., by \emph{Macaulay2}.
    \end{proof}

    \section*{Appendix 2: Macaulay2 code}
    
    This is the Macaulay2 code used for the calculations in Theorems \ref{thm32}, \ref{thm33} and 
    Corollary \ref{cor35}:
    \bigskip
    
    \begin{verbatim}
    R = ZZ/2[a,b,c,x,y,z, Degrees=>{{1,0},{1,0},{1,0},{0,1},{0,1},{0,1}}]
    m1 = R^{{0,3},{1,2},{1,2},{1,2}}
    phi1 = map( m1, R^{{-1,2}}, matrix{{a*x+b*y+c*z},{a^2},{b^2},{c^2}})
    phi2 = map( R^{{1,4}}, m1, matrix{{a*x+b*y+c*z,x^2,y^2,z^2}})
    E = trim( (ker phi2) / (image phi1) )
    E1 = resolution E
    B = transpose generators ker transpose presentation E
    A = ideal(a*x+b*y+c*z)
    
    w = matrix{{b+c,b+c,x+z,a+c,z,y,1}}
    S = ideal(w * B)
    
    sat = i -> saturate(saturate(i, ideal(a,b,c)), ideal(x,y,z))
    sing = i -> sat( minors( codim i, jacobian(i)) + i)
    
    codim S                                             -- should be 2
    SingS = sing S                                    -- should be ideal 1
    
    AS = sat(A + S)
    for i from 0 to (numgens AS)-1 do if degree AS_i=={3,0} then p1 = ideal(AS_i)
    for i from 0 to (numgens AS)-1 do if degree AS_i=={0,3} then p2 = ideal(AS_i)
    C1 = sat(p1 + S) : AS
    C2 = sat(p2 + S) : AS
    minimalPrimes C1
    minimalPrimes C2
    \end{verbatim}

\bibliographystyle{amsplain}
\bibliography{mybib}
	\end{document}